\documentclass[a4paper,11pt]{article}
\usepackage{amsmath,amsthm,amssymb,amscd,esint,epsfig, mathrsfs,graphics,color}
\usepackage{verbatim}
\usepackage[english]{babel}
\newtheorem{thm}{Theorem}
\newtheorem{lem}[thm]{Lemma}
\newtheorem{prop}[thm]{Proposition}

\def\R{\mathbb R}
\def\C{\mathbb C}

\def\D{\mathbb D}

\def\cR{\mathcal R}
\def\cK{\mathcal K}

\def\F{\mathcal F}

\def\B{\mathcal B}

\def\supp{\operatorname{supp}}
\def\d{\partial_z}
\def\dbar{\partial_{\overline{z}}}
\def\Id{\mathbf{Id}}

\title{\sc{Beltrami equations with coefficient \\in the fractional Sobolev space $W^{\theta, \frac2{\theta}}$} 
\footnotetext{\hspace{-0.35cm}
2010 \emph{Mathematics Subject Classification}. 30C62, 35J46, 42B20, 42B37
\endgraf
{\it Key words and phrases}. 
Quasiconformal mapping, Beltrami equation, Fractional Sobolev spaces, Beltrami operators.
}}
\author{\it Antonio L. Bais\' on, Albert Clop and Joan Orobitg.}
\date{}

\begin{document}
\maketitle

\begin{abstract}
In this paper, we look at quasiconformal solutions $\phi:\C\to\C$ of Beltrami equations
$$
\dbar \phi(z)=\mu(z)\,\d \phi (z).
$$
where $\mu\in L^\infty(\C)$ is compactly supported on $\D$, $\|\mu\|_\infty<1$ and belongs to the fractional Sobolev space $W^{\alpha, \frac2\alpha}(\C)$. Our main result states  that
$$\log\d\phi \in W^{\alpha, \frac2\alpha}(\C)$$
whenever $\alpha>\frac12$. Our method relies on an $n$-dimensional result, which asserts the compactness of the commutator 
$$[b,(-\Delta)^\frac{\beta}{2}]:L^\frac{np}{n-\beta p}(\R^n)\to L^p(\R^n)$$
between the fractional laplacian $(-\Delta)^\frac\beta2$ and any symbol $b\in W^{\beta,\frac{n}\beta}(\R^n)$, provided that $1<p<\frac{n}{\beta}$.
\end{abstract}

\section{Introduction}

\noindent

\noindent
A Beltrami coefficient is a function $\mu\in L^\infty(\C)$ with $\|\mu\|_\infty<1$. By the well-known Measurable Riemann Mapping Theorem, to every compactly supported Beltrami coefficient $\mu$ one can associate a unique homeomorphism $\phi:\C\to\C$ in the local Sobolev class $W^{1,2}_{loc}$ such that the \emph{Beltrami equation} 
$$\dbar\phi(z) = \mu(z)\,\d\phi(z)$$ 
holds for almost every $z\in\C$, and at the same time, $|\phi(z)-z|\to 0$ as $|z|\to\infty$. One usually calls $\phi$ the principal solution, and it is known to be a $K$-quasiconformal map with $K=\frac{1+\|\mu\|_\infty}{1-\|\mu\|_\infty}$, since
$$|\dbar\phi(z)|\leq \frac{K-1}{K+1}|\d\phi(z)|  \qquad\text{at almost every $z\in\C$.}$$
Recent works have shown an interest in describing the Sobolev smoothness of $\phi$ in terms of that of $\mu$. As noticed already at \cite{CFMOZ}, remarkable differences are appreciated under the assumption $\mu \in W^{\alpha,p}$,  
depending on if $\alpha p< 2$, $\alpha p=2$ or $\alpha p>2$. In this paper, we focus our attention on the case $\alpha p=2$.\\
\\
It was proven at \cite{CFMOZ} that if $\mu\in W^{1,2}$ then $\phi$ belongs to the local Sobolev space $W^{2,2-\epsilon}_{loc}$ for each $\epsilon>0$ 
(and further one cannot take $\epsilon=0$ in general). The proof was based on the elementary fact that
\begin{equation}\label{cfmozlog}
\mu\in W^{1,2}\hspace{1cm}\Rightarrow\hspace{1cm}\log(\d\phi)\in W^{1,2}.
\end{equation}
In particular, $\log\d\phi$ enjoys a slightly better degree of smoothness than $\d\phi$ itself. 
It is a remarkable fact that this better regularity cannot be deduced only from the fact that $\d\phi\in W^{1,2-\epsilon}_{loc}$ for every $\epsilon>0$. 
Somehow, this means that $\log\d\phi$ contains more information than $\d\phi$.\\ 
\\
Similar phenomenon had been observed much earlier in the work of Hamilton \cite{DH}, where it is shown that
\begin{equation}\label{hamilton}
\mu\in VMO\hspace{1cm}\Rightarrow\hspace{1cm}\log(\d\phi)\in VMO.
\end{equation}
Again, the $VMO$ smoothness of $\log(\d\phi)$ cannot be completely transferred to $\d\phi$ itself. Indeed, the example $\phi(z)=z\,(\log|z|-1)$, 
in a neighbourhood of the origin, has $VMO$ Beltrami coefficient (at least locally) but clearly $D\phi\notin VMO$. 
\\
The $VMO$ setting is interesting in our context since it can be seen as the limiting space of $W^{\alpha,\frac2\alpha}$. 
Certainly, the complex method of interpolation shows that
$$[VMO, W^{1,2}]_\alpha=W^{\alpha,\frac2\alpha}, \hspace{1cm}0<\alpha<1$$ 
(see for instance \cite{RR}). Thus, it is natural to ask if a counterpart to implication \eqref{cfmozlog} holds in $W^{\alpha,\frac2\alpha}$. In the present paper, we prove the following theorem.

\begin{thm}\label{Teo1}
Let $\alpha\in(\frac{1}{2} , 1)$. Let $\mu$ be a Beltrami coefficient with compact support and such that $\mu\in W^{\alpha,\frac{2}{\alpha}}(\C)$. Let $\phi$ be the principal solution to the $\C$-linear Beltrami equation
$$\dbar\phi\,=\,\mu\,\d\phi\,.$$
Then, $\log\left(\partial\phi\right)\in W^{\alpha,\frac{2}{\alpha}}\left(\C\right)$.
\end{thm}

\noindent
The proof of Theorem \ref{Teo1} is based on two facts. The first one is the following a priori estimate for linear Beltrami equations with coefficients belonging to $W^{\alpha, \frac2\alpha}(\C)$.

\begin{thm}\label{apriori}
Let $\alpha\in (0,1)$ and $1<p<\frac2\alpha$. Let $\mu,\nu$ be a pair of Beltrami coefficients with compact support, such that $\||\mu|+|\nu|\|_\infty\leq k<1$ and  $\mu,\nu\in W^{\alpha,\frac{2}{\alpha}}(\C)$. For every $g\in W^{\alpha,p}(\C)$ the equation
$$\dbar f - \mu\,\d f -\nu\,\overline{\d f}=g$$
admits a solution $f$ with $Df\in W^{\alpha p}(\C)$, unique modulo constants, and such that the estimate
$$\|Df\|_{W^{\alpha,p}(\C)}\leq C\,\|g\|_{W^{\alpha,p}(\C)}$$
holds for a constant $C$ depending only on $k$, $\|\mu\|_{W^{\alpha, \frac2\alpha}(\C)}$ and $\|\nu\|_{W^{\alpha, \frac2\alpha}(\C)}$.
\end{thm}

\noindent
Theorem \ref{apriori} is sharp, in the sense that one cannot take $p=\frac2\alpha$. Thus, Theorem \ref{Teo1} shows that $\log\d\phi$ enjoys better regularity than $\d\phi$ itself. \\
\\
The study of logarithms of derivatives of quasiconformal maps goes back to the work of Reimann \cite{R}, 
where it was shown that the real-valued logarithm $\log|\d\phi|\in BMO$ whenever $\|\mu\|_\infty<1$. 
References involving the complex logarithm $\log\d\phi$ also lead to \cite{Ahl}. More recently, in \cite{AIPS} the authors obtained sharp bounds for the $BMO$ norm of $\log\d\phi$ also with the only assumption $\|\mu\|_\infty<1$.\\
\\
The second main ingredient in the proof of Theorem \ref{Teo1} is a compactness result for commutators of pointwise multipliers and the fractional laplacian, which holds in higher dimensions and has independent interest. In order to state it, given a measurable function $u:\R^n\to\R$ we denote
\begin{equation}\label{derifrac}
 D^\beta u (x):=
\lim_{\epsilon\to 0}
C_{n,\beta}
\int_{|x-y|>\epsilon}\frac{u(x)-u(y)}{|x-y|^{n+\beta}}\,dy.
\end{equation}
This is a principal value representation of the fractional laplacian $(-\Delta)^\frac{\beta}2$, whose symbol at the Fourier side is
$$
\widehat{D^\beta u}(\xi)= \widehat{(-\Delta)^{\frac{\beta}2}u}(\xi)=|\xi|^\beta\,\hat{u}(\xi).
$$
The operator $D^\beta$ can also be seen as the formal inverse of $I_\beta$, the classical Riesz potential of order $\beta$, which can be represented as
$$\widehat{I_\beta u}(\xi)=|\xi|^{-\beta}\,\hat{u}(\xi).$$
With this notation, a function $u$ belongs to $W^{\beta,p}$, $1<p<\infty$, if and only if $u$ and $D^\beta u$ belong to $L^p$, with the corresponding equivalent norm. Analogously, $u\in \dot{W}^{\beta,p}$ if and only if $D^\beta u\in L^p$.
\\
Let us remind that if $T$ and $S$ are two operators, one usually calls $[T, S]=T\circ S - S\circ T$ the \emph{commutator} of $T$ and $S$. 

\begin{thm}\label{cpct}
Let $\beta\in(0,1)$ and $b\in W^{\beta,\frac{n}{\beta}}(\R^n)$. Then, the commutator 
$$
[b, D^\beta] : L^\frac{np}{n-\beta p}(\R^n)\to\,L^p(\R^n)\\
$$
is bounded and compact whenever $1<p<\frac{n}{\beta}$.
\end{thm}

\noindent
The boundedness of the commutator can be seen as a consequence of fractional versions of the Leibnitz rule. For the compactness, the Fr\'echet-Kolmogorov characterization of compact subsets of $L^p$ is combined with the cancellation properties of the kernel of the commutator.  Also, in the proof of Theorem \ref{Teo1} one uses Theorem \ref{cpct} with $\beta=1-\alpha$. This explains the restriction $\alpha>\frac12$ in Theorem \ref{Teo1}, as what one really uses is that $\mu\in W^{1-\alpha,\frac2{1-\alpha}}(\C)$. Note that this space contains $W^{\alpha, \frac2\alpha}(\C)$ if and only if $\alpha>\frac12$.\\
\\
A detailed proof of Theorem \ref{cpct} is provided at Section \ref{compactness}. In Section \ref{beltr}, we find a priori estimates for generalized Beltrami equations with coefficients in $W^{\theta,\frac2\theta}$, and  prove Theorem \ref{Teo1} and Theorem \ref{apriori}. \\
\\
\textbf{Acknowledgements}. The three authors are partially supported by the projects 2014SGR75 (Generalitat de Catalunya), MTM2013-44699-P (Ministerio de Econom\'\i a y Competitividad) and Marie Curie Initial Training Network MAnET (FP7-607647).
A. Clop is also supported by the Programa Ram\'on y Cajal.

\section{Proof of Theorem \ref{cpct}}\label{compactness}

The proof of Theorem \ref{cpct} we present here is based on classical ideas, see for instance \cite{KLi}. We will need the following auxilliary result about the Leibnitz rule for fractional derivatives.

\begin{prop}\label{KVP}\emph{(Kenig-Ponce-Vega's Inequality \cite{KPV})}  \newline
Let $\beta\in\left(0,1\right)$ and $1<p<\frac{n}{\beta}$. Then the inequality
$$
\|D^\beta(f\,g)-f\,D^\beta g\|_p\leq C\,\|D^\beta f\|_\frac{n}{\beta}\,\|g\|_\frac{np}{n-\beta p}.
$$
holds whenever $f, g\in \mathcal{C}^\infty_c(\R^n)$. 
\end{prop}

\noindent
With this result at hand, we immediately get that the commutator
$$[b,D^\beta]:L^\frac{np}{n-\beta p}(\R^n)\to L^p(\R^n)$$
admits a unique bounded extension. Remarkably,
$$
\|[b,D^\beta]\|_{L^\frac{np}{n-\beta p}(\R^n)\to L^p(\R^n)}\leq C\,\|b\|_{\dot{W}^{\beta,\frac{n}{\beta}}(\R^n)}.$$
As a consequence, if $b_n\in\mathcal{C}^\infty_c(\R^n)$ is such that
$$\lim_{n\to\infty}\|b_n-b\|_{\dot{W}^{\beta,\frac{n}{\beta}}(\R^n)}=0$$
then
$$\lim_{n\to\infty}\|[b_n,D^\beta]-[b,D^\beta]\|_{L^\frac{np}{n-\beta p}(\R^n)\to L^p(\R^n)}=0$$
Thus, we are reduced to prove Theorem \ref{cpct} with the extra assumption $b\in\mathcal{C}^\infty_c(\R^n)$.  To this end, we observe that the commutator $C_b=[b,D^\beta]$ can be represented as an integral operator
$$
\aligned
C_bf(x)
&=b(x)\,P.V.\int K(x,y)\,(f(x)-f(y))\,dy-P.V.\int K(x,y)\,(f(x)\,b(x)-b(y)\,f(y))\,dy\\
&=P.V. \int K(x,y)\,(b(y)-b(x))\,f(y)\,dy\\
&=\int \cK(x,y)\,f(y)\,dy
\endaligned$$
where 
$$\cK(x,y)= C_{n,\beta}\,\frac{(b(y)-b(x))}{|y-x|^{n+\beta}}$$
and the principal value has been removed from the last integral because the smoothness of $b$ ensures that $x\mapsto \cK(x,y)$ is integrable. For $C_b$ to be compact, we need to prove that the image under $C_b$  of the unit ball of $L^\frac{np}{n-\beta p}(\R^n)$ is compact in $L^p(\R^n)$. To this end, we denote
$$\F=\{C_bf: \; \|f\|_{L^\frac{np}{n-\beta p}(\R^n)}\leq 1\}.$$
The classical Fr\'echet-Kolmogorov's Theorem asserts that $\F$ is relatively compact if and only if the following conditions hold:
\begin{itemize}
\item[$(i)$] $\F$ is uniformly bounded, i.e. $\sup_{\psi\in\F}\|\psi\|_{L^p(\R^n)}<\infty$.
\item[$(ii)$] $\F$ vanishes uniformly at $\infty$, i.e. $\sup_{\psi\in\F}\|\psi\,\chi_{|x|>R}\|_{L^p(\R^n)}\to 0$ as $R\to\infty$.
\item[$(iii)$] $\F$ is uniformly equicontinuous, i.e. $\sup_{\psi\in\F}\|\psi(\cdot+h)-\psi(\cdot)\|_{L^p(\R^n)}\to 0$ as $|h|\to 0$. 
\end{itemize}
In our particular case, every element $\psi\in\F$ has the form $\psi=C_bf$ with $\|f\|_{L^\frac{np}{n-\beta p}(\R^n)}\leq 1$. 
Thus $(i)$ follows automatically from the boundedness of $[b, D^\beta]:L^\frac{np}{n-\beta p}(\R^n)\to L^p(\R^n)$.\\ 
\\
\noindent
To prove $(ii)$, let $R_0>0$ be such that $\supp(b)\subset B(0,R_0)$. At points $x$ with $|x|> 3R_0$ we have
\begin{equation}\label{decay}
| C_bf(x)| \le \int \frac{|f(y)\, b(y)|}{|x-y|^{n+\beta}}\, dy \le C\frac{\| b\|_\infty}{|x|^{n+\beta}}\int_{B(0,R_0)}
|f(y)|\, dy \le C\frac{\| b\|_\infty}{|x|^{n+\beta}} \|f\|_{q}R_0^{n\frac{q-1}{q}}.
\end{equation}
Thus, if $R> 3R_0$ then
$$
\aligned
\int_{|x|> R}|C_b f(x)|^p\,dx \le C_R \| b\|_{\infty}^{p} \| f \|^{p}_{\frac{np}{n-\beta p}}
\int_{|x|> R} |x|^{-p(n+\beta)}\, dx \to 0 \qquad \text{as } R \to \infty
\endaligned$$
as needed. \\
\\
For the proof of $(iii)$, we could proceed as usually, which means to regularize the kernel $\cK$ in the diagonal $\{ x=y\}$. Then we would prove the compactness of this regularization and finally the limit of compact operators would give us the result. However, a more direct approach is available, since $ \|\cK (x,\cdot)  \|_{L^1(\R^n)}$ is uniformly bounded. 

\begin{lem}
One has
\begin{equation}\label{equicont}
\lim_{h\to 0}\sup_{f\neq 0}\frac{\|C_bf(\cdot+h)-C_bf(\cdot)\|_{L^q(\R^n)}}{\|f\|_{L^q(\R^n)}}=0
\end{equation}
whenever $1\leq q\leq \infty$.
\end{lem}
\begin{proof}
We start by observing that
$$
\aligned
\|\cK (x,\cdot)  \|_{L^1(\R^n)}
&=\int _{|x-y|\le 1} |\cK (x,y)| \,dy + \int _{|x-y|> 1} |\cK (x,y)| \,dy\\
&\le C \|\nabla b\|_\infty  \int _{|x-y|\le 1} |x-y|^{-n-\beta+1} \,dy +
   C \| b\|_\infty  \int _{|x-y|> 1} |x-y|^{-n-\beta} \,dy\\
&\le C\left\{ \frac{\|\nabla b\|_\infty}{1-\beta} + \frac{\| b\|_\infty}{\beta}     \right\} := A
\endaligned$$
As a consequence, the behavior of $C_{b}f$ is like the convolution of the function $f$ with a $L^{1}$-kernel. In particular, by Jensen's inequality one gets
 \begin{equation}\label{convo}
 \| C_bf\|_{q}\le A \| f\|_{q}, \qquad 1\le q\le \infty ,
 \end{equation}
so that $C_b:L^q(\R^n)\to L^q(\R^n)$, $1\leq q\leq \infty$. \\ 
\\
Towards \eqref{equicont}, we need to estimate the translates of $C_b$. Clearly,
$$
\aligned
\| C_{b}f(\cdot +h) &- C_{b}f(\cdot)  \|_{q}^q
=\int \left| \int f(y) (\cK (x+h,y) - \cK(x,y)) \,dy \right|^{q} \,dx\\
&\leq\int\left(\int  |f(y)|^q\,|\cK (x+h,y) - \cK(x,y)| \,dy  \right)\,\left(\int  |\cK (x+h,y) - \cK(x,y)| \,dy \right)^\frac{q}{q'}dx\\
&\leq (2A)^{q-1}\,\int \left(\int |\cK (x+h,y) - \cK(x,y)| \,dx\right)\, |f(y)|^q\,dy\\
&= (2A)^{q-1}\,B(h)\,\int|f(y)|^qdy\endaligned$$
where $B(h)=\sup_y \|\cK(\cdot+h,y)-\cK(\cdot, y)\|_{L^1(\R^n)}$. In order to find estimates for $B(h)$, we choose an arbitrary $\rho>0$ and write $$
\int | \cK (x+h,y) - \cK(x,y) |\, dx = \int_{|x-y|\le \rho} \cdots+ \int_{|x-y|> \rho} \cdots := I + II.
$$
The integrability of $\cK$ gives that $I$ is small if $\rho$ is small enough. Indeed,
$$
\int_{|x-y|\le \rho} |\cK(x,y) |\, dx \le \|\nabla b\|_\infty  \int _{|x-y|\le \rho} |x-y|^{-n-\beta+1} \,dx =
C \frac{\|\nabla b\|_\infty}{1-\beta} \rho^{1-\beta} .
$$
Moreover, if $x\in B(y,\rho)$ then $x+h\in B(y, \rho+|h|)$ so that
$$
\int_{|x-y|\le \rho} |\cK(x+h,y) |\, dx \le \int_{|x-(y-h)|\le 2\rho} |\cK(x+h,y) |\, dx\le
C \frac{\|\nabla b\|_\infty}{1-\beta} (\rho+|h|)^{1-\beta} .
$$
Therefore, there exists $\rho_0>0$ such that if $\rho<\rho_0$ and $|h|<\rho_0/2$ then $I\le \varepsilon/((2 A)^{q-1})$. Let us then fix $\rho=\rho_0/2$, and take care of $II$. Note that, since $|h|< \rho_0/2$ and $|x-y|> \rho$, we have
$$
\aligned
|\cK (x,+hy) - \cK(x,y)| & = \left| (b(y)-b(x+h))\left( \frac{1}{|x+h-y|^{n+\beta}} -  \frac{1}{|x-y|^{n+\beta}} \right) \right.\\
& \left. \qquad + \frac{1}{|x-y|^{n+\beta}} (b(x)-b(x+h)) \right|\\
& \le 2\| b\|_{\infty} \frac{C|h|}{|x-y|^{n+\beta +1}}
+ \| \nabla b\|_{\infty} \frac{|h|}{|x-y|^{n+\beta}}
\endaligned$$
Then, since we fixed $\rho=\rho_0/2$, 
$$
\aligned
II &\le C\| b\|_{\infty} |h| \int_{|x-y|>\rho} \frac{dx}{|x-y|^{n+\beta +1}}
+ C\| \nabla b\|_{\infty} |h| \int_{|x-y|>\rho} \frac{dx}{|x-y|^{n+\beta}}\\[3mm]
& \le C \frac{|h|}{\beta}\left( \frac{\| b\|_{\infty}}{\rho_0^{1+\beta}} + \frac{\|\nabla b\|_{\infty}}{\rho_0^{\beta}}  \right).
\endaligned
$$
Thus, by taking $|h|$ sufficiently small, we see that $II\le \varepsilon/((2 A)^{q-1})$. Hence $B(h)\to 0$ as $|h|\to 0$, and thus \eqref{equicont} follows.
\end{proof}

\noindent
With the above Lemma, the proof of $(iii)$ is almost immediate. Indeed, by \eqref{decay} we see that 
$$\aligned
\| C_{b}f(\cdot +h)- C_{b}f(\cdot)  \|_{p}^p & = \int_{|x|\le R} | C_{b}f(x +h)- C_{b}f(x) |^p\, dx\\
& \qquad + \int_{|x|> R} | C_{b}f(x +h)- C_{b}f(x) |^p\, dx\\
& \le \| C_{b}f(\cdot +h)- C_{b}f(\cdot)  \|_{\frac{np}{n-\beta p}}^p R^{\beta p}\\
& \qquad +  C_R \| b\|_{\infty}^{p} \| f \|^{p}_{\frac{np}{n-\beta p}}\int_{|x|> R} |x|^{-p(n+\beta)}\, dx.
\endaligned$$
at least for $R>3R_0$. In particular, the last term is small if $R$ is large enough. But for this particular $R$, and using \eqref{equicont}, the penultimate term is also small if $|h|$ is small. Therefore $(iii)$ follows. Theorem \ref{cpct} is proved.

\section{Beltrami operators in fractional Sobolev spaces}\label{beltr}

\noindent
The regularity theory for Beltrami equations relies on the behavior of the Beurling operator, which is formally defined as a principal value operator,
$$
\B f(z)= - \frac{1}{\pi}\,   \text{p.v.}  \int_{\mathbb C} f(z-w)\frac{1}{w^2}\, dA(w).
$$
This operator intertwines the $\d$ and $\dbar$ derivatives. More precisely, its Fourier representation
$$
\widehat{\B f}(\xi) = \frac{\overline{\xi}}{\xi}\,\, \hat{f}(\xi).
$$
makes it clear that $\B(\dbar f)=\d f$, at least when $f$ is smooth and compactly supported. Furthermore, $\B$ is an isometry on $L^2(\C)$,  and as a Calder\'on-Zygmund operator, it can be boundedly extended to $ L^{p}(\C)$ whenever $1<p<\infty$. \\
\\
Before proving Theorem \ref{Teo1}, we first state and prove the following fact about generalized Beltrami equations. Let us recall that $\overline{\B}$ denotes the composition of $\B$ with the complex conjugation operator, that is, $\overline{\B}(f) = \overline{\B (f)}$.

\begin{prop}\label{invertib}
Let $\alpha\in (0,1)$. Let $\mu, \nu \in W^{\alpha,\frac2\alpha}(\C)$ be  compactly supported Beltrami coefficients, with $\| |\mu| + |\nu| \|_{\infty}\leq k<1$. Then the generalized Beltrami operators
$$
 \Id-\mu\,\B  -\nu \overline{\B}:\dot{W}^{\alpha, p}(\C)\to \dot{W}^{\alpha, p}(\C)
$$
are bounded and boundedly invertible if $1<p<\frac2\alpha$. 
\end{prop}

\begin{proof} 
The operators $\Id-\mu\,\B-\nu\,\overline{\B}$ are clearly bounded in $\dot{W}^{\alpha,p}(\C)$, since $\B$ preserves $\dot{W}^{\alpha,p}(\C)$  (recall that we are assuming $1<p<\frac2\alpha$) and also because if $\mu\in L^\infty(\C)\cap W^{\alpha,\frac2\alpha}(\C)$ then $\mu$  is a pointwise multiplier of $\dot{W}^{\alpha,p}(\C)$ (similarly for $\nu$). This fact follows directly working on the expression \eqref{derifrac} for $D^\alpha$ or see \cite[p. 250]{RuSic}. Also, the operator $\Id-\mu\,\B-\nu\,\overline{\B}$ is clearly injective in $\dot{W}^{\alpha,p}(\C)$, as its kernel is a subset of $L^\frac{2p}{2-\alpha p}(\C)$ were 
we already know it is injective (see \cite{I} for a proof in the $\C$-linear setting, and \cite{Ko} or also \cite{CC} for a proof in the general case). Thus, in order to get the surjectivity (and finish the proof by the Open Mapping Theorem) we will prove that $\Id-\mu\,\B-\nu\,\overline{\B}$ is a Fredholm operator on $\dot{W}^{\alpha,p}(\C)$ with index $0$. To do this, it is sufficient if we prove that
$$D^\alpha(\Id-\mu\,\B-\nu\,\overline{\B})I_\alpha:L^p(\C)\to L^p(\C)$$
is a Fredholm operator of index $0$, since both properties stay invariant under the topological isomorphisms
$$
\aligned
D^\alpha&: \dot{W}^{\alpha,p}(\C)\to L^p(\C),\\
I_\alpha&:L^p(\C)\to \dot{W}^{\alpha,p}(\C).
\endaligned$$
But this follows easily. Indeed,
$$\aligned
D^\alpha(\Id-\mu\,\B-\nu\,\overline{\B})I_\alpha 
&=\Id-D^\alpha(\mu\,\B+\nu\,\overline{\B})I_\alpha \\
&=\Id-\mu\,\B-\nu\,\overline{\B} -[D^\alpha,\mu]\,\B\,I_\alpha -[D^\alpha,\nu]\,\overline{\B}\,I_\alpha\\
\endaligned$$
Above, $\Id-\mu\,\B-\nu\,\overline{\B}$ is invertible in $L^p(\C)$ by \cite{I}. Also, $[D^\alpha,\mu]\,\B\,I_\alpha$ is the composition of the bounded operators $I_\alpha:L^p(\C)\to L^\frac{2p}{2-\alpha p}(\C)$ and $\B:L^\frac{2p}{2-\alpha p}(\C)\to L^\frac{2p}{2-\alpha p}(\C)$ with the operator $[D^\alpha,\mu]:L^\frac{2p}{2-\alpha p}(\C)\to L^p(\C)$, which is compact by Theorem \ref{cpct}. Hence $[D^\alpha,\mu]\,\B\,I_\alpha:L^p(\C)\to L^p(\C)$ is compact, and the same happens to $[D^\alpha,\nu]\,\overline{\B}\,I_\alpha$. Thus the term on the right hand side is the sum of an invertible operator with two compact operators. Hence it is a Fredholm operator. The claim follows. 
\end{proof}

\noindent
We are now ready to prove Theorem \ref{apriori}. 

\begin{proof}[Proof of Theorem \ref{apriori}]
By simplicity, we assume that $\nu=0$. Otherwise, the proof follows similarly. First of all, let us observe that if $g\in \dot{W}^{\alpha,p}(\C)$ and $\alpha p<2$ then automatically $g\in L^\frac{2p}{2-\alpha p}(\C)$ 
by the Sobolev embedding. On the other hand, and since $W^{\alpha,\frac2\alpha}(\C)\subset VMO$, we know from \cite{I} that a solution 
$f\in \dot{W}^{1,\frac{2p}{2-\alpha p}}(\C)$ exists, and moreover
$$\|Df\|_{L^\frac{2p}{2-\alpha p}(\C)}\leq C\,\|g\|_{L^\frac{2p}{2-\alpha p}(\C)}\leq C\,\|g\|_{\dot{W}^{\alpha,p}(\C)}.$$
Our goal consists of replacing the term on the left hand side by $\|Df\|_{\dot{W}^{\alpha,p}(\C)}$. \\
\\
To do this, we first note that $\d f = \B(\dbar f)$, since $f\in \dot{W}^{1,\frac{2p}{2-\alpha p}}$. Thus \eqref{beltramiequation} is equivalent to
$$
(\Id-\mu\,\B)(\dbar f) = g
$$
Now, from Proposition \ref{invertib} and our assumption $g\in \dot{W}^{\alpha,p}(\C)$, we also know that there is a unique $F\in \dot{W}^{\alpha,p}(\C)$  such that
\begin{equation}\label{first}
(\Id-\mu\,\B) F= g
\end{equation}
for which we know the estimate $\|F\|_{\dot{W}^{\alpha,p}(\C)}\leq C\,\|g\|_{\dot{W}^{\alpha,p}(\C)}$ holds. Of course, by the Sobolev embedding, $F\in L^\frac{2p}{2-\alpha p}(\C)$. From the invertibility of $\Id-\mu\,\B$ on $L^\frac{2p}{2-\alpha p}(\C)$, we immediately get that $F=\dbar f$ almost everywhere, and therefore $\dbar f\in \dot{W}^{\alpha,p}(\C)$. Proving that $\d f\in\dot{W}^{\alpha,p}(\C)$ is very easy, as we already knew that $f\in\dot{W}^{1,\frac{2p}{2-\alpha p}}(\C)$ and so we can be sure that $\d f = \B(\dbar f)$. Thus, $Df\in \dot{W}^{\alpha,p}(\C)$ and certainly
$$\|Df\|_{\dot{W}^{\alpha,p}(\C)}\leq C\, \|F\|_{\dot{W}^{\alpha,p}(\C)}\leq C\, \|g\|_{\dot{W}^{\alpha,p}(\C)}$$
as desired. 
\end{proof} 
 
\noindent
Towards the proof of Theorem \ref{Teo1}, we denote by $\mathsf C(h)$ the solid Cauchy transform,
\begin{equation}\label{cauchy}
 \mathsf C \, h(z) = \frac{1}{\pi}\,  \int_{\mathbb C} h(z-w)\frac{1}{w}\, dA(w).
\end{equation}
This operator appears naturally as a formal inverse to the $\dbar$ derivative, that is, the formula $\dbar \mathsf C(h) =h$ holds if $h\in L^p(\C)$ and $1<p<\infty$. Another important feature about the Cauchy transform is that $\partial \mathsf C= \B$. The Cauchy and Beurling transforms allow for a nice representation of the principal solution $\phi$ of the Beltrami equation $\dbar\phi=\mu\,\d\phi$,
$$
\phi(z) = z +  \mathsf C(h)(z),
$$
see for instance \cite[p. 165]{AIM}. In this representation, $h$ is a solution to the integral equation
$$
(\Id-\mu\,\B)(h) = \mu.
$$
As a consequence, the invertibility of the Beltrami operators $\Id-\mu\,\B$ also plays a central role in determining the smoothness of $\phi$. In particular, by applying Proposition \ref{invertib} with $\mu\in W^{\alpha,\frac2\alpha}$, we see that $Dh\in W^{\alpha, p}$ provided that $p<\frac2\alpha$, whence $D\phi\in W^{\alpha,p}_{loc}$. As a consequence, by Stoilow's Factorization Theorem (e.g., \cite[section 5.5]{AIM}), the same conclusion holds for any quasiregular solution $f$ of
$\dbar f -\mu\,\d f=0.$ However, this is not enough for Theorem \ref{Teo1}, which we prove now. 

\begin{proof}[Proof of Theorem \ref{Teo1}]
We will first prove that if $\mu\in W^{\alpha,\frac2\alpha}(\C)$ is a compactly supported Beltrami coefficient and $\alpha>\frac12$ (this is the point where
we use that restriction) the operator
$$T_\mu:=I_{1-\alpha}\left(\Id-\mu\,\B\right)D^{1-\alpha}: L^\frac{2}{\alpha}\left(\C\right) \longmapsto L^\frac{2}{\alpha}\left(\C\right)$$
is continuously invertible, with lower bounds depending only on $\|\mu\|_{L^\infty(\C)}$ and $\|\mu\|_{W^{\alpha,\frac2\alpha}(\C)}$. 
To do this, we proceed as usually,
$$\begin{aligned}
T_\mu=I_{1-\alpha}(\Id-\mu\,\B)D^{1-\alpha}
&=\Id-I_{1-\alpha}\mu \B D^{1-\alpha} \\
&=\Id-\mu\,\B+I_{1-\alpha}\,[D^{1-\alpha},\mu ]\,\B.
\end{aligned}$$
Here, the term $\Id-\mu\,\B$ is bounded and continuously invertible in $L^\frac2\alpha(\C)$ by \cite{I}. Concerning the second term on the right hand side, from $\mu\in W^{\alpha,\frac2\alpha}(\C)\cap L^\infty(\C)$ and $\frac12<\alpha$ we easily get that $\mu\in W^{1-\alpha,\frac2{1-\alpha}}(\C)$. Thus we are legitimate to use Theorem \ref{cpct} with $\beta=1-\alpha$ and $p=\frac2\alpha$ and get that $[\mu, D^{1-\alpha}]$ is a compact operator from $L^\frac2\alpha(\C)$ into $L^2(\C)$.  As a consequence, we obtain that $T_\mu$ is a Fredholm operator from $L^\frac2\alpha(\C)$ into itself, which clearly has index $0$. So the desired lower bounds will be automatic if we see that it is injective.\\
\\
Let $F\in L^\frac{2}{\alpha}$ such that $T_\mu (F)=0$.  We want to show that $F=0$. First, if $F\in \dot{W}^{1-\alpha,2}(\C)$ then the result follows easily. Indeed, we can then write $F:=I_{1-\alpha}f$ for some $f\in L^2$ and write the equation in terms of $f$. We get  $I_{1-\alpha}(\Id-\mu\,\B)f=0$. From the classical $L^2$ theory, we have that $f=0$ and hence $F=0$. For a general $F\in L^\frac{2}{\alpha}$ satisfying $T_\mu(F)=0$ we will prove that necessarily $F\in \dot{W}^{1-\alpha,2}(\C)$, and therefore $F=0$. To do this, again we decompose $T_\mu$ in terms of the commutator,
$$(\Id-\mu\,\B)F= I_{1-\alpha}[\mu,D^{1-\alpha}]\B F.$$
Then by Theorem \ref{cpct} the term on the right hand side above belongs to $\dot{W}^{1-\alpha,2}(\C)$, because 
 $F\in L^\frac2\alpha(\C)$. Using again that $\alpha>\frac12$ one has $\mu\in W^{1-\alpha,\frac2{1-\alpha}}(\C)$, and therefore we can use Proposition \ref{invertib} to get that  $\Id-\mu\,\B: \dot{W}^{1-\alpha,2}(\C)\to\dot{W}^{1-\alpha,2}(\C)$ is continuously invertible. Hence 
$$F= (\Id-\mu\,\B)^{-1} I_{1-\alpha}[\mu,D^{1-\alpha}]\B F$$
belongs to $\dot{W}^{1-\alpha, 2}(\C)$.  The claim follows. 
\\
We now finish the proof. Given $\mu\in W^{\alpha,\frac2\alpha}(\C)$, we approximate it by $\mu_n\in\mathcal{C}^\infty_c(\C)$ in the $W^{\alpha,\frac2\alpha}(\C)$ topology, in such a way that $\|\mu_n\|_{L^\infty(\C)}\leq\|\mu\|_{L^\infty(\C)}$. Then every $\mu_n$ admits a principal quasiconformal map $\phi_n$, for which the function $g_n=\log \d\phi_n$ is well defined and solves
$$\dbar g_n-\mu_n\,\d g_n=\d\mu_n.$$
Therefore
$$(\Id-\mu_n\B)\dbar g_n=\d\mu_n.$$
We use the Fourier representation of the classical Riesz transforms in $\R^2$,
$$\widehat{\mathcal{R}_j u}\left(\xi\right)=-i\frac{\xi_j}{\left|\xi\right|}\widehat{u}\left(\xi\right) \qquad j=1,2$$
to represent 
$$\aligned
\dbar g &=-\pi D^{1-\alpha} (\cR_1+i\cR_2)(D^\alpha g)\\
\d g &=-\pi D^{1-\alpha}(\cR_1-i\cR_2)(D^\alpha g).
\endaligned$$
As a consequence, we obtain
$$(\Id-\mu_n\,\B) D^{1-\alpha}(\cR_1+i\cR_2) (D^\alpha g_n)=D^{1-\alpha}(\cR_1-i\cR_2)(D^\alpha \mu_n),$$
and therefore
$$T_{\mu_n}(\cR_1+i\cR_2) (D^\alpha g_n)=(\cR_1-i\cR_2)(D^\alpha \mu_n).$$
We recall that both $\cR_1+i\cR_2$ and $\cR_1-i\cR_2$ are bounded and continuously invertible operators in $L^p(\C)$, $1<p<\infty$. Moreover, we have just seen that $T_{\mu_n}$ is boundedly invertible in $L^\frac2\alpha(\C)$ with bounds depending only on $\|\mu_n\|_{L^\infty(\C)}$ and $\|\mu_n\|_{W^{\alpha,\frac2\alpha}(\C)}$. However, each $\|\mu_n\|_\infty$ (and respectively $\|\mu_n\|_{W^{\alpha,\frac2\alpha}(\C)}$) is bounded by a constant multiple of $\|\mu\|_\infty$ (respectively $\|\mu\|_{W^{\alpha,\frac2\alpha}(\C)}$). Hence
$$\aligned
\|g_n\|_{\dot{W}^{\alpha,\frac2\alpha}(\C)}
&=\|D^\alpha g_n\|_{L^\frac2\alpha(\C)}\\
&\leq C(\alpha)\,\|(\cR_1+i\cR_2)D^\alpha g_n\|_{L^\frac2\alpha(\C)}\\
&\leq C\left(\alpha,\|\mu\|_{L^\infty(\C)},\|\mu\|_{W^{\alpha,\frac2\alpha}(\C)}\right)\,\|T_{\mu_n}(\cR_1+i\cR_2) (D^\alpha g_n)\|_{L^\frac2\alpha(\C)}\\
&\leq C\left(\alpha,\|\mu\|_{L^\infty(\C)},\|\mu\|_{W^{\alpha,\frac2\alpha}(\C)}\right)\,\|(\cR_1-i\cR_2) D^\alpha \mu_n\|_{L^\frac2\alpha(\C)}\\
&\leq C\left(\alpha,\|\mu\|_{L^\infty(\C)},\|\mu\|_{W^{\alpha,\frac2\alpha}(\C)}\right).
\endaligned$$
It then follows that $g_n$ is a bounded sequence in $\dot{W}^{\alpha,\frac2\alpha}(\C)$. By the Banach-Alaoglu theorem there exists $h\in \dot{W}^{\alpha,\frac2\alpha}(\C)$ such that
$$\lim_{n\to\infty}\langle g_n,\varphi\rangle=\langle h,\varphi\rangle$$
for each $\varphi\in W^{-\alpha,\frac2{2-\alpha}}(\C)$. Remarkably, by the weak lower semicontinuity of the norm,
$$
\|h\|_{\dot{W}^{\alpha,\frac2\alpha}(\C)}=\|D^\alpha h\|_{L^\frac2\alpha(\C)}\leq \liminf_{n\to\infty}\|D^\alpha g_n\|_{L^\frac2\alpha(\C)}\leq C\left(\alpha,\|\mu\|_{L^\infty(\C)},\|\mu\|_{W^{\alpha,\frac2\alpha}(\C)}\right).
$$
Incidentally, we already knew from the classical theory that $\phi_n$ converges in $W^{1,p}_{loc}(\C)$ to the principal quasiconformal map $\phi$ associated to $\mu$. In particular, modulo subsequences, $\d\phi_n$ converges to $\d\phi$ almost everywhere. But then $g_n$ converges almost everywhere to $\log(\d\phi)$. It then follows that $\log(\d\phi)=h$ and so we deduce that $\log(\d\phi)$ belongs to $\dot{W}^{\alpha,\frac2\alpha}(\C)$, with the same bound than $h$. The theorem follows.
\end{proof}

\noindent
A. L. Bais\'on, A. Clop, J. Orobitg\\
Departament de Matem\`atiques\\
Universitat Aut\`onoma de Barcelona\\
08193-Bellaterra (Catalonia)

\end{document}